\documentclass{article}
\usepackage[utf8]{inputenc}
\usepackage{multicol}
\usepackage{amsmath,amsfonts,amssymb,amsthm,url}
\usepackage{tikz}
\usepackage{longtable}
\usepackage{booktabs}
\usepackage{multirow}
\usepackage{authblk}
\usetikzlibrary{shapes,arrows}
\usetikzlibrary{backgrounds}

\usepackage{tabularray}
\usepackage{caption}
\usepackage{subcaption}

\usepackage{hyperref}

\usepackage{longtable}
\usepackage{mathtools}

\usepackage{multicol} 

\usepackage{algorithm}
\usepackage{algpseudocode}
\usepackage{mdframed}

\def\cF{\mathcal{F}} 
\def\cL{\mathcal{L}} 
\def\Nat{\mathbb{N}} 
\def\Real{\mathbb{R}} 
\def\Acc{\Gamma} 
\def\Poly{\mathcal{S}} 
\def\Mat{\mathcal{M}} 
\def\flag{\mathcal{F}} 
\def\field{\mathbb{F}}
\def\kos{\mathbb{K}}
\def\spcs{\mathcal{V}} 

\def\cl{\mathrm{cl}}

\def\cip{\texttt{CI}}

\def\gep{\texttt{GE}}

\def\dlp{\texttt{DL}} 
\def\akp{\texttt{AK}}
\def\clp{\texttt{CL}} 

\def\sss{\smallsetminus}

 \newcommand\fullv[1]{#1}
 \newcommand\confv[1]{}

\theoremstyle{plain}
\newtheorem{proposition}{Proposition}[section]
\newtheorem{theorem}[proposition]{Theorem}

\newtheorem{lemma}[proposition]{Lemma}

\theoremstyle{definition}
\newtheorem{definition}[proposition]{Definition}

\newtheorem{example}[proposition]{Example}

\newtheorem{lpp}[proposition]{Linear Programming Problem}

\makeatletter
\let\oldlt\longtable
\let\endoldlt\endlongtable
\def\longtable{\@ifnextchar[\longtable@i \longtable@ii}
\def\longtable@i[#1]{\begin{figure}[t]
\onecolumn
\begin{minipage}{0.5\textwidth}
\oldlt[#1]
}
\def\longtable@ii{\begin{figure}[t]
\onecolumn
\begin{minipage}{0.5\textwidth}
\oldlt
}
\def\endlongtable{\endoldlt
\end{minipage}
\twocolumn
\end{figure}}
\makeatother

\title{Optimizing Extension Techniques for Discovering Non-Algebraic Matroids}

\author{Michael Bamiloshin}
\author{Oriol Farr\`as}

\affil{Universitat Rovira i Virgili, Tarragona, Spain}

\begin{document}
\maketitle

\begin{abstract}
         In this work, we revisit some combinatorial and information-theoretic extension techniques for detecting non-algebraic matroids. These are the Dress-Lov\'asz and Ahlswede-K\"orner extension properties. We provide optimizations of these techniques to reduce their computational complexity, finding new non-algebraic matroids on 9 and 10 points. In addition, we use the Ahlswede-K\"orner extension property to find better lower bounds on the information ratio of secret sharing schemes for ports of non-algebraic matroids. 
\end{abstract}
\makeatletter{\renewcommand*{\@makefnmark}{}
	\footnotetext{
\fullv{This work is the full version of \cite{BaFa24}.} The authors are supported by the project HERMES, funded by INCIBE and by the European Union NextGeneration EU/PRTR, and the project ACITHEC PID2021-124928NB-I00, MCIN/AEI/10.13039/ 501100011033/FEDER, EU. Additionally, Oriol Farr\`as is supported by grant 2021 SGR 00115 from the Government of Catalonia. 
    }
 \makeatother}
\section{Introduction}
\label{sec: intro}

The characterization of matroids that admit a linear representation over a field is a natural problem that was formulated in the early stages of matroid theory. This notion of linear representation can be extended to algebraic representation over field extensions, considering the rank function determined by the transcendence degree instead of the linear dimension. Matroids that admit such a representation are said to be algebraic, and the characterization of these matroids is the main objective of this work. In addition to linear and algebraic, other kinds of matroid representations, such as multilinear (or folded-linear) and entropic (or by partitions), have also been studied. It is known that linearly representable matroids are multilinear, and multilinear matroids are entropic. The class of algebraic matroids contains the class of linear matroids, but it does not contain the class of multilinear matroids~\cite{Ben16}. And the class of entropic matroids does not contain the class of algebraic matroids~\cite{Matus99}.

The study of linearly representable matroids has attracted a lot of interest and has applications in different areas, such as information theory, cryptography (secret sharing schemes), and coding theory (network coding). See~\cite{BLP08,BrDa91,DFZ07,SHL08}, for example. \fullv{Recently, a connection between algebraic matroids and secret sharing schemes was presented in \cite{BFM25}.}Results on the other matroid representation classes mentioned above have also been applied in these areas~\cite{Ben16, ESG10, Matus99, Sey92, SiAs98}. 

Ingleton and Main presented in~\cite{InMa75} a necessary condition for a matroid to be algebraic: in a \emph{full} algebraic matroid of rank at least 4, if there are three pairwise but not all coplanar lines, then all three lines have a common intersection. This result was later generalized by Lindstr{\"o}m~\cite{Lindstrom1988}.

The rank function of a full linear matroid is modular; but it is not necessarily so for full algebraic matroids. Nevertheless, a similar combinatorial property to modularity was shown for full algebraic matroids by Dress and Lov\'asz~\cite{DrLo87}: for every pair of flats in a full algebraic matroid, there exists a flat called the \emph{quasi-intersection} (denoted \emph{pseudo-intersection} in~\cite{BjLo86}) that simulates their intersection. The Ingleton-Main lemma and its generalizations can also be viewed as extension properties of algebraic matroids, similar to the Euclidean and generalized Euclidean intersection properties of linear matroids~\cite{BaWa89, Bam21, BFP23}. 

In~\cite{Bollen18}, Bollen did extensive work on the problem of matroid algebraicity. Using Frobenius flocks, he found some matroids on 9 points that are not algebraic over fields of characteristic 2; and using a recursive implementation of the Ingleton-Main lemma, he was able to discover many matroids on 9 points that are not algebraic over any field.
The Ingleton-Main lemma loses its efficacy when applied to sparse paving matroids with rank greater than 4, as such matroids will always satisfy the property. Also, some matroids might be Frobenius-flock representable and still not be algebraic.

While the aforementioned techniques emanate from matroid theory works, other techniques coming from information theory can also be applied to the problem of matroid classification.
These include the common information ($\cip$), the Ahlswede-K\"orner ($\akp$), and the copy lemma ($\clp$) properties. 

$\cip$ is an extension property of linear polymatroids that is used to show that a matroid does not admit a linear representation. It was used (sometimes, implicitly) in many works studying the classification of linearly representable matroids, e.g.~\cite{BBFP21,Ing71,MaRo08}.

$\akp$ is a property of almost entropic polymatroids, and was used to find non-Shannon information inequalities. Since algebraic matroids are almost entropic~\cite{Matus24}, we bring $\akp$ to the core of techniques for discovering non-algebraic matroids. 

The direct application of both the $\akp$ and $\cip$ techniques to the linear programming problems for finding lower bounds on the information ratio of secret sharing schemes was introduced in~\cite{FKMP20}. Later, $\clp$ was also applied to this problem~\cite{GuRo19}, and improved lower bounds were found in~\cite{Bam21,BBFP21,Gurpinar22}. 

\subsection{Our Contributions}

In this work, we revisit some combinatorial and information-theoretic tools for detecting matroids that do not admit an algebraic representation. We provide optimizations for techniques based on the Dress-Lov\'asz and Ahlswede-K\"orner extension properties. 
Similar to~\cite{Bam21,BBFP21,Bollen18}, we show recursive applications of these techniques. Finding new results on matroid extensions, we are able to reduce the computational cost of using these techniques.

Applying these optimized techniques, we find new non-algebraic and non-almost entropic matroids on 9 and 10 points. We continue the classification work of algebraic matroids of Bollen~\cite{Bollen18}, completing the classification of $(4,9)$ matroids (i.e., matroids with rank 4 on 9 points) that do not satisfy the Dress-Lov\'asz property at recursive depths smaller than 8. 
For $(5,9)$ matroids, we show some smallest sparse paving matroids that satisfy the Ingleton inequality, have rank greater than 4, and are not almost entropic. These particular matroids would not have been found using Frobenius flocks, as they are Frobenius-flocks representable, nor using the Ingleton-Main lemma due to the fact that they are sparse paving. Additionally, we show an identically self-dual, sparse paving rank-5 matroid on 10 points that is not almost entropic. This matroid has both the Tic-Tac-Toe matroid and its dual as minors.

Another contribution of this work is an improvement on the linear programming technique for finding lower bounds on the information ratio of secret sharing schemes presented in~\cite{Csi97,FKMP20}. We define an LP problem that uses a more restrictive property of almost entropic polymatroids~\cite{BFP23}, which is also a consequence of the $\akp$ lemma. Using this approach, we obtain improved lower bounds for ports of matroids that are not algebraic.

For the interested reader, the programs used in this paper are available at \href{https://github.com/bmilosh/algebraic-matroids-extensions}{https://github.com/bmilosh/algebraic-matroids-extensions}.

\subsection{Organization}

The organization of the rest of this paper is the following. In Section~\ref{sec: mats and polys}, we introduce the notations we use in this work. In Section~\ref{subs: reps}, we talk about relationships between different ways of representing a matroid. We introduce the extension properties we are focused on in Section~\ref{sec:AlgbComb}. We encounter the first set of optimizations in this work in Section~\ref{sec:PropsOfAKPolys}. The next set of optimizations come in Sections~\ref{sub: AK opts} and~\ref{sub: DL opts}. We present some non-algebraic matroids we found in Section~\ref{sec: exps}, and finish with some results on secret sharing in Section~\ref{sec: sss results}. 

\section{Matroids and Polymatroids}
\label{sec: mats and polys}

We refer the reader to~\cite{Oxley11} for an in-depth discussion of matroids, and to~\cite{BaWa89, Bam21, BFP23} for matroid extension properties and techniques. 

\begin{definition}\label{df:polym1}
	Given a finite set $Q$ and a function
	$f\colon\mathcal{P}(Q)\rightarrow \Real$, 
	the pair $(Q,f)$ is called a \textit{polymatroid} if
	the following properties are satisfied for all $X,Y\subseteq Q$.
	\begin{description}
		\item[(P1)] $f(\emptyset) = 0$.
		\item[(P2)]{\label{P2}} $f(X) \leq f(Y)$ if $X \subseteq Y$.
		\item[(P3)]{\label{P3}} $f(X \cap Y) + f(X \cup Y) \le f(X) + f(Y)$.
	\end{description}
	The set $Q$ and the function $f$ are, respectively,
	the \textit{ground set} and the \textit{rank function} of the polymatroid.
	The rank function of an \emph{integer} polymatroid only takes integer values.
	A \emph{matroid} is an integer polymatroid  $(Q,r)$ 
	such that $r(X) \le |X|$ for every $X \subseteq Q$. For compactness, given any sets $X,Y\subseteq Q$, we write the union $X\cup Y$ as $XY$, $r(X;Y)$ to denote $r(X) + r(Y) - r(XY)$, and $r(X|Y)$ for $r(XY)-r(Y)$.
\end{definition}

For a polymatroid $\Poly = (Q,f)$ and a set $B \subseteq Q$,
the \emph{deletion} $\Poly \setminus B$ of $B$ from $\Poly$
is the polymatroid $(Q \sss B, \widehat{f})$
with $\widehat{f}(X) = f(X)$ for every $X \subseteq Q \sss B$,
while the \emph{contraction} $\Poly/B = (Q \sss B, \widetilde{f})$ of $B$ from $\Poly$ is such that $\widetilde{f}(X) = f(XB) - f(B)$ for every $X \subseteq Q \sss B$. 
Every polymatroid that is obtained from $\Poly$ by 
applying deletions and contractions is called a \emph{minor} of $\Poly$.
Finally, observe that minors of matroids are matroids.

Consider sets $Q,Z$ with $Q\cap Z=\emptyset$. A polymatroid $(QZ,g)$ is an \emph{extension} of
a polymatroid $(Q,f)$ if they satisfy that $g(X)=f(X)$ for every $X\subseteq Q$.

Let $\Mat = (Q,r)$ be a matroid. 
The \emph{independent sets} of $\Mat$ 
are the sets $X \subseteq Q$ with $r(X) = |X|$.
Every subset of an independent set is independent. 
The \emph{bases} of $\Mat$ are its maximal independent sets, and its minimal dependent sets are called \emph{circuits}.
All bases have the same number of elements,
which equals $r(Q)$, the \emph{rank} of the matroid.
A set $X \subseteq Q$ is a \emph{flat} of $\Mat$
if $r(Xx) > r(X)$ for every $x \in Q \sss X$. 
The flats with rank $r(Q)-1$ are called \emph{hyperplanes}. Flats $X,Y\subseteq Q$ are \emph{modular} if $r(X) + r(Y) = r(XY) + r(X\cap Y)$, and \emph{nonmodular} otherwise. 
In addition to the one given in Definition~\ref{df:polym1},
there are other equivalent sets of axioms characterizing matroids
which are stated in terms of the properties
of the independent sets, the circuits, 
the bases, or the hyperplanes.

In a \emph{simple} matroid,
all sets with one or two elements are independent.
A matroid of rank $k$ is \emph{paving}
if the rank of every circuit is either $k$ or $k - 1$.
It is \emph{sparse paving} if,
in addition, all circuits of rank $k - 1$ are flats,
which are called \emph{circuit-hyperplanes}.
The \emph{dual} of $\Mat = (Q,r)$ 
is the matroid $\Mat^* = (Q,r^*)$ with
$r^*(X) = |X| - r(Q) + r(Q \sss X)$
for every $X \subseteq Q$.
Equivalently, $\Mat^*$ is the matroid
on $Q$ whose bases are the complements of the bases of $\Mat$. 


\begin{definition}\label{def:modular cuts}
	Given a matroid $\Mat=(Q,r)$, a \emph{modular cut} $\mathcal{F}$ of $\Mat$ is a family of flats of $\Mat$ satisfying the following properties:
	\begin{enumerate}
		\item For every $F_1\in\mathcal{F}$ and for every flat $F_2$ such that $F_1\subseteq F_2$, $F_2\in\mathcal{F}$, i.e., $\mathcal{F}$ is monotone increasing.
		\item For every modular pair $F_1,F_2\in\mathcal{F}$, $F_1\cap F_2\in\mathcal{F}$, i.e., $\mathcal{F}$ is closed under intersection of modular pairs. 
	\end{enumerate}
\end{definition}

Every proper point extension of a matroid (i.e., an extension by a rank-1 element) corresponds to a modular cut and vice versa~\cite[Section 7.2]{Oxley11}. The modular cut \emph{generated} by the flats $\{F_1,F_2,\ldots,F_k \}$, for some $k>0$ is simply the smallest modular cut that contains these flats. In general, we denote the modular cut generated by a flat $X$ as
    $ \mathcal{F}_X = \{F\subseteq Q:X\subseteq F\text{ and $F$ is a flat of }\Mat \}.$ 

\subsection{Matroid Representations}
\label{subs: reps}

\begin{definition}
A matroid $\Mat=(Q,r)$ is \emph{$\ell$-linearly representable} over a field $\field$ for some $\ell\in\Nat$ if there exists a vector space $\spcs$ and a vector subspace collection $(V_x)_{x\in Q}$ defined over $\field$ with $V_x\subseteq \spcs$ such that 
	$$\dim\left(\sum_{x\in A}V_{x}\right)=\ell\cdot r(A)\text{ for every }A\subseteq Q.$$
\end{definition}
If $\ell=1$, then $\Mat$ is simply said to be linearly representable. Matroids that are $\ell$-linearly representable for $\ell>1$ are said to be \emph{multilinear} (or \emph{folded-linear}) matroids. While all linear matroids are multilinear, the reverse is not true. Examples of multilinear matroids that are not linear were shown in, e.g.,~\cite{PeVZ13,SiAs98}.

Consider an extension $\kos$ of $\field$. An element $x$ of $\kos$ is said to be \emph{algebraic} over $\field$ if it is the root of some non-trivial polynomial in $\field$. Otherwise it is \emph{transcendental}. Given a subset $X\subseteq\kos$, an element $y$ of $\kos$ is \emph{algebraically independent} with respect to $X$ if it is transcendental over the field $\field(X)$. A subset $X\subseteq\kos$ is \emph{algebraically independent} over $\field$ if every element $x\in X$ is algebraically independent with respect to the set $X\setminus x$ (i.e., no element $x\in X$ is the root of some non-trivial polynomial in $\field(X\setminus x)$), and \emph{algebraically dependent} otherwise. The \emph{transcendence degree} of $\kos$ over $\field$ is the size of the largest algebraically independent subset of $\kos$ over $\field$.

\begin{definition}
    A matroid $\Mat=(Q,r)$ is \emph{algebraically representable} over a field $\field$ if there exist an extension $\kos$ of $\field$ and a sequence of elements $(e_i)_{i\in Q}\subseteq \kos$ such that, for every $A\subseteq Q$, $$r(A)=\deg_{tr}\field((e_i)_{i\in A}),$$ where $\field((e_i)_{i\in A})$ is the smallest subfield of $\kos$ containing $\field$ and $(e_i)_{i\in A}$, and $\deg_{tr}$ is the transcendence degree of $\field((e_i)_{i\in A})$ over $\field$.
\end{definition}    
    

    A matroid $\Mat$ whose ground set $Q$ consists of all elements of $\kos$ and is such that $F\subseteq Q$ is a flat of $\Mat$ if and only if $\field((e_i)_{i\in F})$ is algebraically closed is called a \emph{full algebraic} matroid.

Fujishige~\cite{Fuj78} observed that, given a set $Q=\{1,\ldots,n \}$ with an associated set of random variables $\{S_1,\ldots,S_n\}$, the entropy function $h:2^Q\rightarrow\Real_{\geq0}$ on this set expressed as $$h(A)=H(S_A)$$ for every $A\subseteq Q$ such that $A\neq\emptyset$, $h(\emptyset)=0$, and $H(S_A)$ is the Shannon entropy of the random variables indexed by $A$, defines the rank function of a polymatroid. Such a polymatroid is called an \emph{entropic polymatroid}. A matroid is said to be \emph{almost entropic} if it is the limit of a sequence of entropic polymatroids, and \emph{entropic} if its rank function is a multiple of the rank function of an entropic polymatroid.

All linear matroids are algebraic, but the converse is not true. Matroids with less than 8 points are linear, and therefore algebraic, with the V\'amos matroid being the first matroid shown to be non-algebraic~\cite{InMa75}. In the other direction, the non-Pappus matroid is an example of an algebraic matroid that is not linear~\cite{Lindstrom83}. Furthermore, not every multilinear matroid is algebraic~\cite{Ben16}, and not every algebraic matroid is multilinear~\cite{Matus99}. Since every multilinear matroid is entropic, then there are entropic matroids that are not algebraic. All algebraic matroids are almost entropic~\cite{Matus24}, but not all almost entropic matroids are algebraic (nor entropic)~\cite[Remarks 4 \& 5]{Matus18}. For a visual depiction of these relationships, see~\cite[Figure 1]{BBFP21}.

\section{Extension Properties of Algebraic Matroids}
\label{sec:AlgbComb}
In this section, we present the techniques we use to find non-algebraic matroids. Our techniques are based on properties satisfied by algebraic matroids: Ahlswede-K\"orner extensions, and Dress-Lov\'asz extensions.

\subsection{Dress-Lov\'asz Extensions}
\label{sec:DLExts}
Every algebraic matroid can be embedded in its full algebraic matroid, and the same holds for linear matroids. However, unlike in the case of full linear matroids where all pairs of flats are modular, pairs of flats of full algebraic matroids are not necessarily modular. Instead, Dress and Lov\'asz~\cite{DrLo87} showed that pairs of flats of full algebraic matroids have what they called a quasi-intersection.

\begin{theorem}\cite[Theorem 1.5]{DrLo87}\label{thm:DrLo87}
    Let $\Mat=(Q,r)$ be a full algebraic matroid. Then for every pair of flats $X,Y\subseteq Q$ of $\Mat$, there exists a flat $T\subseteq X$ such that, for every flat $X'$ contained in $X$, 
    \[ T\subseteq X'\, \text{ if and only if }\, r(Y|X')=r(Y|X). \]
    Moreover, $X$ and $Y$ are modular if and only if $T=X\cap Y$.
\end{theorem}

One can deduce a necessary condition for a matroid to be algebraic from the Dress-Lov\'asz result as follows. If $\Mat$ is an algebraic matroid for which for some pair of flats ($X,Y$) there is no $T\subseteq X$ satisfying Theorem~\ref{thm:DrLo87}, then $\Mat$ admits a series of proper point extensions in which this flat $T$ exists. 

\begin{definition}\label{def:QI}
    Let $\Mat=(Q,r)$ be a matroid and let $X,Y\subseteq Q$ be a nonmodular pair of flats of $\Mat$. The \emph{quasi-intersection of $(X,Y)$} is a flat $T\subseteq Q$ satisfying the following conditions
    \begin{description}
        \item[(DL1)] $r(T|X)=0$, and
        \item[(DL2)] $r(T|X')=0$ iff $r(Y|X')=r(Y|X)$ for every flat $X'\subseteq X$.
    \end{description}
    Note that the quasi-intersections of $(X,Y)$ and $(Y,X)$ are not necessarily the same.
\end{definition}

\begin{lemma}\label{lem:QI-Unique}
    For a pair of flats $(X,Y)$ in a matroid $(Q,r)$, if the quasi-intersection exists, it is unique.
\end{lemma}
\begin{proof}
    Suppose $T_1,T_2\subseteq Q$ are two quasi-intersections of $(X,Y)$. Then by (DL2), we have $r(T_1|T_2)=0=r(T_2|T_1)$ and therefore, $r(T_1)=r(T_1T_2)=r(T_2)$. Since $T_1$ and $T_2$ are flats, it implies that $T_1=T_2$.
\end{proof}

\begin{definition}\label{def:DLExt}
    Let $\Mat=(Q,r)$ be a matroid and let $X,Y\subseteq Q$ be a nonmodular pair of flats of $\Mat$. A matroid $(QZ,r)$ is a \emph{Dress-Lov\'asz ($\dlp$) extension} of $\Mat$ for $(X,Y)$ if there exists a set $T\subseteq QZ$ that is the quasi-intersection of $(X,Y)$. 
\end{definition}

\begin{definition}\label{def:DLProp}
    A matroid $\Mat=(Q,r)$ satisfies the \emph{Dress-Lov\'asz property ($\dlp$)} if for every nonmodular pair of flats $(X,Y)$ of $\Mat$ there is a $\dlp$-extension. 
\end{definition}

Like the matroid extension properties studied in~\cite{Bam21,BFP23}, we can also give a recursive definition to the Dress-Lov\'asz extension property as follows: 

\begin{definition}
    A matroid $\Mat$ is \emph{1-$\dlp$} if for every pair of flats $X,Y$, there is a $\dlp$-extension. It is \emph{$k$-$\dlp$} for some $k>1$ if for every pair of flats $X,Y$, there is a $\dlp$-extension that is $(k-1)$-$\dlp$. Algebraic matroids are $k$-$\dlp$ for every $k\geq1$ due to~\cite{Matus24}. 
\end{definition}

\subsection{Ahlswede-K\"orner Lemma}
\label{sec: AK lemma}

The Ahlswede-K\"orner lemma describes a property of pairs of information sources~\cite{AhKo77,AhKo06} that also holds for almost entropic polymatroids. 
We use it to find non-algebraic matroids, as in~\cite{BBFP21}, because algebraic matroids are almost entropic~\cite{Matus24}.

In this work, we use a stronger statement of this property that was proved in~\cite[Proposition 3.14]{FKMP20}. Instead of dealing with triples of sets, the $\akp$ property can still be determined using pairs of sets with extra conditions. The following is a formalization of that result and was recently defined in~\cite{BFP23}.

\begin{definition}\label{def:NewAK}
For a polymatroid $(Q,f)$ and sets 
$X, Y \subseteq Q$,
an extension $(Q Z, g)$ of $(Q,f)$ is 
an \emph{Ahlswede-K\"orner extension},
or \emph{{\akp} extension}, for the pair 
$(X, Y)$ if the following conditions are satisfied:
\begin{description}

\item[(AK1)]
$g(Z|X) = 0$, and
\item[(AK2)]
$g(X' | Z)=g(X' |Y)$ 
for every $X' \subseteq X$.
    
\end{description}
Following~\cite{BBFP21}, we call $Z$ the \emph{$\akp$-information} of $(X,Y)$ and denote it $\akp(X,Y)$. We say that a polymatroid satisfies the {\akp} property if for every $(X, Y)$ there exists an {\akp} extension.
\end{definition}
Note that it can sometimes happen that there is a set $Z\subseteq X$ that satisfies the listed conditions. In such a case, by an abuse of notation, we still call such a set an $\akp$ information of $(X,Y)$. The following result shows when this situation may arise.

\begin{lemma}\label{lem: Z iff ModPair}
    Let $\Mat=(Q,r)$ be a matroid, and let $X,Y\subseteq Q$ be flats. There exists a set $Z\subset Q$ that satisfies (AK1) and (AK2) for $(X,Y)$ if and only if $(X,Y)$ is a modular pair.
\end{lemma}
\begin{proof}
    The converse statement is straightforward (just take $Z=X\cap Y)$ and its proof is therefore omitted.
    For the forward direction, note that setting $X'=Z$ in (AK2) gives $r(Z|Y)=0$ and setting $X'=X$ gives $r(Z)=r(X;Y)$. Hence, $Z=X\cap Y$, and $r$ is therefore modular on $(X,Y)$.
\end{proof}

We conclude this section presenting a property of linear polymatroids, the \emph{common information} property ($\cip$). Though not a direct property of algebraic matroids, we introduce it here because we will use the fact that it implies the $\akp$ property~\cite{FKMP20} in some of our results.

\begin{definition}\label{def:CI}
For a polymatroid $(Q,f)$ and sets 
$X, Y \subseteq Q$,
an extension $(Q Z, g)$ of $(Q,f)$ is 
a \emph{common information extension},
or \emph{{\cip} extension}, for the pair 
$(X, Y)$ if the following conditions are satisfied:
\begin{description}

\item[(CI1)]
$g(Z|X) = g(Z|Y) = 0$, and
\item[(CI2)]
$g(Z)=g(X;Y)$.
    
\end{description}
Likewise, we call $Z$ the \emph{common information} of $X$ and $Y$ and denote it $\cip(X,Y)$. We say that a polymatroid satisfies the {\cip} property if for every $(X, Y)$ there exists a {\cip} extension.
\end{definition}

As shown in~\cite[Proposition 3.17]{BFP23}, the $\akp$ property is preserved by minors.
And, from~\cite[Proposition 3.16]{FKMP20} (or, alternatively, \cite[Proposition 3.13]{BFP23}), we have the following relationship between the $\cip$ and $\akp$ properties. We add its proof for the sake of completeness. 

\begin{proposition}\label{prop:CIFromAK}
    Let $\Mat=(Q,r)$ be a matroid, let $X,Y\subseteq Q$ be a nonmodular pair of flats of $\Mat$, and let $x_o=\akp(X,Y)$. Then $r(x_o)=r(X;Y)$. Hence, $x_o=\akp(Y,X)$ if and only if $x_o$ is the common information of $X$ and $Y$. 
\end{proposition}
\begin{proof}
If $x_o=\akp(X,Y)$, then $r(x_o|X)=0$ and $r(Xx_o)=r(X)$ by (AK1). By (AK2),
\[0=r(X|x_o)-r(X|Y)=r(Xx_o)-r(x_o)-r(XY)+r(Y)=r(X;Y)-r(x_o),\]
proving the first statement. If $x_o=\akp(Y,X)$, then $r(x_o|Y)=0$, and $x_o$ satisfies (CI1). Conversely, a common information of $(X,Y)$ satisfies (AK1) and (AK2) for $(X,Y)$ and $(Y,X)$.
\end{proof}

Combining Lemma~\ref{lem: Z iff ModPair} and Proposition~\ref{prop:CIFromAK}, we have that, for flats $X$ and $Y$, $\akp(X,Y)=\akp(Y,X)$ if and only if $r$ is modular on $X$ and $Y$.

\section{Optimizing $\akp$ for Polymatroids}
\label{sec:PropsOfAKPolys}
We observed that, similar to $\cip$, checking if a matroid is $\akp$ can be restricted to only the flats of the matroid as opposed to checking all possible subsets of the ground set of the matroid. This fact is presented in the following theorem, which is proved later in this section. 

\begin{theorem}\label{thm:AKRedefined}
A polymatroid $(Q,f)$ satisfies the $\akp$ property if for every pair of flats $(X,Y)$, there is an extension $(QZ,f)$ that satisfies conditions (AK1) and (AK2'), where 
\begin{description}
\item[(AK2')] $\,\,\, f(X' | Z)=f(X' |Y)$ for every flat $X' \subseteq X$.
\end{description}
\end{theorem}

As a consequence of this theorem, we can check the existence of $\akp$ extensions of a polymatroid with the following linear program. Note that while this is defined for a single extension, it can easily be extended to an arbitrary number of extensions of the polymatroid. 

\begin{lpp}
\label{tab:LPAKold}
Given a polymatroid $(Q,f)$, check if for every pair of flats $(U,V)$ there exists a polymatroid extension $(Qx_0,f)$ that satisfies (AK1) and (AK2').
\end{lpp}

Theorem~\ref{thm:AKRedefined} is proved in two steps. First, we show in Lemma~\ref{prop:AK2reduced} that (AK2) can be reduced to (AK2'). Then in Proposition~\ref{prop:AKForFlatsNew}, we show that it is enough to check (AK1) and (AK2) for pairs of flats. But first, the following results are some of the properties of polymatroids that will be used frequently (sometimes implicitly) in our discussions on $\akp$.

\begin{lemma}\label{lem:closure}
    Let $\Poly=(Q,f)$ be a polymatroid. For any $X,Y\subseteq Q$, $$f(XY)=f(\bar{X}Y)=f(\bar{X}\bar{Y}),$$
    where $\bar{X}$ and $\bar{Y}$ are the closures of $X$ and $Y$, respectively.
\end{lemma}
\begin{proof}
It is enough to prove that $f(XY)=f(\bar{X}Y).$
By monotonicity, $f(\Bar{X}Y) \geq f(XY)$. By submodularity, we have that
    \begin{align*}
        f(XY) + f(\Bar{X}) &\geq f(XY\Bar{X}) + f(XY\cap\Bar{X}) = f(\Bar{X}Y) + f(X\cup (\Bar{X}\cap Y)) \\
        &\geq f(\Bar{X}Y) + f(X) = f(\Bar{X}Y) + f(\Bar{X}).
    \end{align*}
Thus, $f(XY)=f(\Bar{X}Y)$.
\end{proof}

\begin{lemma}\label{lem:ClosureExtension}
    Let $\Poly=(Q,f)$ be a polymatroid and let $\Poly'=(Q Z,f)$ be an extension of $\Poly$. For any $U\subseteq Q$, $V\subseteq Q Z$ and $\Bar{U}=\cl_{\Poly}(U)$, $$f(UV)=f(\Bar{U}V).$$
\end{lemma}
\begin{proof}
    Observe that
    $f(UV)\leq f(\Bar{U}V)\leq f(\cl_{\Poly'}(U)V).$
     By Lemma~\ref{lem:closure}, the first and last terms above are equal, proving the result.
\end{proof}

In the next result, we show that one does not in fact need to check property (AK2) of Definition~\ref{def:NewAK} for every subset $X'$ of $X$; it is sufficient to check it only for subsets of $X$ that are flats. 

\begin{lemma}\label{prop:AK2reduced}
    Let $\Poly=(Q,f)$ be a polymatroid and let $\Poly'=(Q Z,f)$ be an extension of $\Poly$. 
    Take $X,Y,\Bar{X}\subseteq Q$ where $\Bar{X}=\cl_{\Poly}(X)$. Then 
    \[ f(X|Z) = f(X|Y)\,\text{ if and only if }\, f(\Bar{X}|Z)=f(\Bar{X}|Y).\]
\end{lemma}
\begin{proof}
    By Lemma~\ref{lem:ClosureExtension}, we have that $f(\Bar{X}Z)=f(XZ)$, and $f(\Bar{X}Y)=f(XY)$ by Lemma~\ref{lem:closure}. Thus,  
    $f(\Bar{X}|Z)=f(\Bar{X}Z) - f(Z) = f(XZ) - f(Z) = f(X|Z)$ and, analogously, $f(\Bar{X}|Y) = f(X|Y)$,
    completing the proof.
\end{proof}
As a consequence of this result, (AK2) is equivalent to (AK2').
We show next that for any matroid, the $\akp$ property can be checked using only flats of the polymatroid.

\begin{proposition}\label{prop:AKForFlatsNew}
    Let $\Poly=(Q,f)$ be a polymatroid, let $X,Y\subseteq Q$, and let $\Bar{X}=\cl(X)$ and $\Bar{Y}=\cl(Y)$. If $\Poly$ admits an $\akp$ extension for $(\Bar{X},\Bar{Y})$ then it admits an $\akp$ extension for $(X,Y)$.
\end{proposition}
\begin{proof}
    Let $Z=\akp(\Bar{X},\Bar{Y})$. Note that $f(Z|X)=0$ by Lemma~\ref{lem:closure}. And for all $X'\subseteq X$, since $f(X'|Z)=f(X'|\Bar{Y})$, then $f(X'|Z)=f(X'|Y)$ again by Lemma~\ref{lem:closure}, completing the proof.
\end{proof}

Hence, if the polymatroid admits an $\akp$ extension for every pair of flats, then it does so for every pair of sets, proving the theorem.

This reduction in the primary number of sets involved in the $\akp$ property from 3 to 2 greatly reduces computation time in checking the property. Given a matroid $(Q,r)$ with flats $\{F_1,F_2,\ldots,F_k\}$ for some integer $k$, while the 3-set formulation of the $\akp$ property involves checking about $k!/(k-3)!$ triples of flats, the 2-set formulation involves about $k!/(k-2)!$ pairs of flats. 

\section{Optimizing $\akp$ and $\dlp$ for Matroids}
\label{sec:OptForAKAndDL}
Admitting an $\akp$ extension for $(X,Y)$ does not necessarily indicate admitting an $\akp$ extension for $(Y,X)$~\cite[Proposition 3.16]{FKMP20}, and it might therefore be worthwhile to apply LP~\ref{tab:LPAKold} for both $(X,Y)$ and $(Y,X)$ simultaneously when checking for $\akp$ extensions. Nevertheless, the results here show that this does not help in some cases. 

It's important to note that results in Section~\ref{sub: AK opts} are \emph{negative} in the sense that they implicitly show us which pairs of sets to avoid when testing for $\akp$. 
This is due to the fact that a polymatroid satisfying the $\akp$ property doesn't necessarily mean it is almost entropic. But on the other hand, if it does not satisfy $\akp$, then we know for sure it is not almost entropic, and therefore, also not algebraic. Hence, it is more practical using it to find polymatroids in the latter category.

The next two results are well-known properties of matroids that we will apply at various times for the rest of this section. 
\begin{lemma}\label{lem:sprspvngcirhyps}
    Let $H_1$ and $H_2$ be distinct circuit-hyperplanes of a matroid $\Mat$ of rank $k$. Then $r(H_1\cap H_2)\leq k-2$. In addition, if $\Mat$ is paving, then $|H_1\cap H_2|\leq k-2$.
\end{lemma}

\begin{proof}
By submodularity,
$$2(k-1)=r(H_1)+r(H_2)\geq r(H_1H_2)+r(H_1\cap H_2) =k+r(H_1\cap H_2),$$
so $k-2\geq r(H_1\cap H_2)$. The second part of the result holds because $H_1\cap H_2$ is an independent set in a paving matroid.
\end{proof}

\begin{lemma}\label{lem:linesAndHyps}
    Let $\Mat=(Q,r)$ be a matroid of rank $k$ and let $X,Y$ be a nonmodular pair of flats of $\Mat$.
    \begin{enumerate}
        \item[(i)] If $X$ is a line, then $r(X|Y) = 1$ and $r(X\cap Y)=0$. 
        \item[(ii)] If $Y$ is a hyperplane, then $r(Y|X')>r(Y|X)$ for every flat $X'$ in $X$.
    \end{enumerate}
\end{lemma}

\begin{proof}
    Since $(X,Y)$ are nonmodular, $r(X)+r(Y)>r(XY)+r(X\cap Y)$ and $r(XY)\geq r(Y)+1$. In (i), using that $r(X)=2$ we get $2+r(Y)>r(Y)+1+r(X\cap Y)$, which implies that $r(X\cap Y)=0$ and $r(XY)=r(Y)+1$. In (ii), if $X'\nsubseteq Y$, then $r(X'Y)=k=r(XY)$ and so $r(Y|X')>r(Y|X)$. If  $X'\subseteq Y$, then 
    $r(X)+r(Y)>r(XY)+r(X\cap Y)\geq r(XY)+r(X')$, which implies $r(Y|X')>r(Y|X)$.
\end{proof}

\subsection{$\akp$}
\label{sub: AK opts}

The results in this section show how we reduce the computational cost of applying $\akp$ to the detection of non-algebraic matroids. We achieve this by identifying combinations of flats of the matroid for which the matroid is guaranteed to have an $\akp$ extension. 

\begin{proposition}\label{prop:XYCircuit}
    Let $\Mat=(Q,r)$ be a matroid of rank $k$. Let $X,Y\subseteq Q$ be a disjoint, nonmodular pair of flats of $\Mat$ such that their union is a circuit of $\Mat$. Then $\Mat$ admits an $\akp$ extension $\Mat'$ for $(X,Y)$. 
\end{proposition}

\begin{proof}
    Let $r(X)=\ell$, $r(Y)=m$ and $|XY|=s$. Since $X$ and $Y$ are disjoint, $s = \ell + m$. 
    Now, let $\Mat'=(Qx_o,r)$ be the proper point extension of $\Mat$ corresponding to the modular cut $\mathcal{F}_X$. We show that $\Mat'$ is an $\akp$ extension for $(X,Y)$. We have that $r(x_o|X)=0$ and $r(X|x_o)=\ell-1$. Also, note that 
    \[ r(X|Y)=r(XY)-r(Y)=(|XY|-1)-r(Y)=s-1-m=\ell-1. \]  
    Any $X'\subset X$ is independent. Therefore, for every $X'\subset X$,
    \[ r(X'|x_o)=r(X'x_o)-1=|X'|+1-1=|X'|\text{ and } \] 
    \[ r(X'|Y)=r(X'Y)-m=|X'|+m-m=|X'|, \]
    completing the proof.
\end{proof}

\begin{lemma}\label{lem:AKExtForLines}
    Let $\Mat=(Q,r)$ be a matroid and let $X,Y\subseteq Q$ be a nonmodular pair of flats of $\Mat$ such that $r(X)=2$. Then $\Mat$ admits an $\akp$ extension for $(X,Y)$.
\end{lemma}

\begin{proof}
    Let $\Mat'=(Qx_o,r)$ be the proper point extension of $\Mat$ corresponding to $\mathcal{F}_X$. By Lemma~\ref{lem:linesAndHyps}, $ r(X|Y)=1$ and $r(X\cap Y)=0$. Hence, for every  flat $X'\subset X$ of rank 1, $r(X'|Y)=1$ and 
    \[ r(X'|x_o) = r(X'x_o) - r(x_o) = 2-1=1\]
    Hence, $\Mat'$ is an $\akp$ extension of $\Mat$ by $x_o$.
\end{proof}

\begin{lemma}\label{lem:YHypAK}
    Let $\Mat=(Q,r)$ be a matroid of rank $k$ and let $X,Y\subseteq Q$ be disjoint flats of $\Mat$. If $Y$ is a hyperplane then $\Mat$ admits an $\akp$ extension for $(X,Y)$.
\end{lemma}

\begin{proof}
    Observe that the rank of the $\akp$-information of $(X,Y)$ is 
    \[ r(X;Y)=r(X)+r(Y)-r(XY)=r(X)+k-1-k=r(X)-1. \]
    Since $\akp$ extensions always exist for modular pairs (see Lemma~\ref{lem: Z iff ModPair}),  
    the case $r(X)=2$ is trivial. Suppose that $r(X)>2$. 
    Consider the matroid $(Qe_1,r)$ corresponding to the modular cut $\flag_X$. Observe that $\cl(X_1e_1)=Xe_1$ for every flat $X_1\subset X$ such that $r(X_1)=r(X)-1$. Next, consider the matroid $(Qe_1e_2,r)$ corresponding to the modular cut $\flag_{Xe_1}$ and see that $\cl(X_2e_1e_2)=Xe_1e_2$ for every flat $X_2\subset X$ such that $r(X_2)=r(X)-2$. It is clear that one can continue this until we have $\cl(X_ce_1e_2\ldots e_c)=Xe_1e_2\ldots e_c$ for $c=r(X;Y)$ and for every flat $X_c\subset X$ such that $r(X_c)=r(X)-c$. Now, set $Z=e_1e_2\ldots e_c$. We have that, for every $X' \subseteq X$, $r(Z|X)=0$ and $r(X'|Z)=r(X'Z)-r(Z)=r(X)-r(Z)=r(X)-c=1=r(X'|Y)$. Hence, $(QZ,r)$ is an $\akp$ extension for $(X,Y)$.
\end{proof}

\begin{lemma}\label{lem:2HypsImAK}
       Let $\Mat=(Q,r)$ be a matroid and let $H_1,H_2\subseteq Q$ be hyperplanes of $\Mat$. Then $\Mat$ admits a $\cip$ extension for $(H_1,H_2)$. 
\end{lemma}

\begin{proof}
    If $H_1$ and $H_2$ are modular, then it is enough to take the extension corresponding to the modular cut generated by their intersection. In the case where they are nonmodular, take $\mathcal{F} = \{H_1,H_2, Q \}$.    
    Since $H_1$ and $H_2$ are nonmodular hyperplanes of $\Mat$, then $\mathcal{F}$ is the smallest modular cut of $\Mat$ generated by $H_1$ and $H_2$, and hence, $\Mat$ admits a proper point extension corresponding to $\mathcal{F}$, and therefore, a $\cip$ extension for $(H_1,H_2)$. 
\end{proof}

This next result takes into account the following fact. If $X$ and $Y$ are flats of a sparse paving matroid of rank $k$ such that $r(X) + r(Y) \leq k$, then $(X,Y)$ is a modular pair. While the result only applies to sparse paving matroids, its importance comes from the fact that sparse paving matroids are conjectured to predominate in any asymptotic enumeration of matroids~\cite{MaWe13,MNWW11,PeVdP14}. Hence, results applying to such matroids will affect almost all matroids, asymptotically.

\begin{proposition}\label{prop:AKSpecialRankGeneral}
     Let $\Mat=(Q,r)$ be a sparse paving matroid of rank $k$. Let $X,Y\subseteq Q$ be a disjoint, nonmodular pair of flats of $\Mat$ such that $r(X)+r(Y)=k+1$ and $X$ is not a circuit-hyperplane. Then $\Mat$ admits an $\akp$ extension with respect to $(X,Y)$.
\end{proposition}

\begin{proof}
    First, note that if $Y$ is a hyperplane (resp. $X$ is a line), then Lemma~\ref{lem:YHypAK} (resp.~\ref{lem:AKExtForLines}) applies. Therefore, since $r(X) + r(Y) = k + 1$, we have $r(X), r(Y) \leq k - 2$. Hence $X$ and $Y$ are independent because $M$ is paving.

    Since $X$ and $Y$ are disjoint independent sets, then $|XY| = |X| + |Y| = k+1$, and since all sets of size greater than $k$ have rank $k$, then $r(XY)=k$.

    Any two subsets $Z_1, Z_2\subset XY$ with $|Z_1|=|Z_2|=k$ have $|Z_1\cap Z_2|=k-1$. Then, by Lemma~\ref{lem:sprspvngcirhyps}, there is at most one circuit-hyperplane in $XY$. If $XY$ is a circuit, then Prop.~\ref{prop:XYCircuit} applies.  
    
    Now consider the case that $XY$ contains a circuit-hyperplane. Let $X_1\subseteq X$ and $Y_1\subseteq Y$ be such that $X_1Y_1$ is a circuit-hyperplane. Without loss of generality, let $X_1 \subset X$ and $Y_1 = Y$. Let $(Qx_o,r)$ be the extension of $\Mat$ corresponding to $\mathcal{F}_{X_1}$. It is clear that (AK1) is trivially satisfied. For (AK2), first, we have 
    $r(X|x_o)=r(X)-1=r(X|Y)$. And for all $X'\subset X$ such that $X'\not=X_1$, we have that, by independence of $X'x_0$ and $X'Y$, $r(X'|x_o)=|X'|=r(X'|Y)$, since no such $X'$ is in $\mathcal{F}_{X_1}$. And finally, for $X_1\subset X$, first note that $r(X_1)=r(X)-1$ and so
    $r(X_1|x_o)=r(X)-2$. Then, observe that $r(X_1|Y)=k-1-r(Y)=r(X)-2$, which concludes the proof.
\end{proof}

\fullv{Proposition~\ref{prop:AKSpecialRankGeneral} has an interesting implication for the TTT matroids introduced in~\cite{BFP23}. If $\Mat$ is a TTT matroid of rank 5 for which the only sets that break $\cip$ are the pairs $(r_i,r_j)$ and $(\ell_i,\ell_j)$ for $i,j=1,2,3$ and $i\not = j$ (see~\cite[Section 4.3]{Bam21} for details of this description), then the matroid is 1-$\akp$. This is because these pairs of flats are such that $r(X)=r(Y)=3$, $X\cap Y=\emptyset$ and $r(X) + r(Y) = k + 1 = 6$.}

From~\cite[Lemma 3.20]{Bam21}, we know that, to check if a matroid $\Mat$ satisfies $\cip$, it is enough to take pairs $(X,Y)$ where $X$ and $Y$ are flats of the matroid. And from~\cite[Proposition 3.15]{FKMP20}, we know that a $\cip$ extension of $\Mat$ for $(X,Y)$ is also an $\akp$ extension of $\Mat$ for the same pair.

Now, if a matroid satisfies the generalized Euclidean intersection ($\gep$) property, then it also satisfies $\cip$. This was first observed in~\cite{BFP23} and we formalize it next.

\begin{proposition}\label{prop:GPImpliesCI}
    Let $\Mat=(Q,r)$ be a matroid. For any nonmodular pair of flats $(X,Y)$ of $\Mat$, if $\Mat$ admits a $\gep$ extension for $(X,Y)$, then it also admits a $\cip$ extension for the same pair.
\end{proposition}

Thus, taking this into account, we see that checking if a matroid satisfies the $\akp$ property can be done even more efficiently by first eliminating the pairs for which the matroid admits a $\gep$ extension. In our experience, this leaves only a very limited number of pairs to check. Then, where possible, one can now apply the other results shown here to these remaining pairs. To illustrate just how much time is saved using this approach, we compare 3 different approaches used to check 1-$\akp$ for the $(5,9)$ matroid with Bollen identifier\footnote{This refers to the way these matroids are organized in the computer programs available at \href{https://github.com/gpbollen/Algebraicity-of-Matroids-and-Frobenius-Flocks/blob/master/Matroid Encyclopedia.ipynb}{https://github.com/gpbollen/Algebraicity-of-Matroids-and-Frobenius-Flocks/blob/master/Matroid Encyclopedia.ipynb}.} $100736$ in Table~\ref{tab:AKTimeTaken}.

\begin{longtblr}[
  caption = {Time taken to check $\akp$ for $(5,9)$ matroid $100736$},
  entry = {Short Caption},
  label = {tab:AKTimeTaken},
]{
  colspec = {ccc}, hlines, 
  row{1} = {font=\bfseries},
}
 $3$-set $\akp$   & $2-$set $\akp$   & $2-$set $\akp$ with $\gep$ Heuristic    \\
$\approx 15$ hours & $\approx 1$ hour & $<1$ minute  \\
\end{longtblr}

\subsection{$\dlp$}
\label{sub: DL opts}

In the case of the $\dlp$ property, we note that, interestingly, all the optimizations shown above for $\akp$ also apply. 
%
%
We state and prove those results here, starting with the $\dlp$ counterpart for Prop.~\ref{prop:XYCircuit}.

\begin{proposition}\label{prop:XYCircuitDL}
    Let $\Mat=(Q,r)$ be a matroid and let $X,Y\subseteq Q$ be a disjoint nonmodular pair of flats such that $XY$ is a circuit. Then $X$ is a quasi-intersection of $(X,Y)$.
\end{proposition}
\begin{proof}
    Here, $r(XY)=r(X)+r(Y)-1$, and so $r(Y|X)=r(Y) + r(X) - 1 - r(X) = r(Y) - 1$. And for every flat $X'\subset X$, $X'Y$ is independent, hence we have $r(Y|X')=r(Y) + r(X') - r(X') = r(Y) > r(Y|X)$, completing the proof.
\end{proof}

This next one combines Lemmas~\ref{lem:AKExtForLines},~\ref{lem:YHypAK} and~\ref{lem:2HypsImAK} into one result for $\dlp$.

\begin{lemma}\label{lem:DLGuaranteed}
    Let $X$ and $Y$ be nonmodular flats of a matroid $\Mat=(Q,r)$. If $X$ is a line or $Y$ is a hyperplane,
    then $X$ is a quasi-intersection of $(X,Y)$ in~$\Mat$.
\end{lemma}
\begin{proof}
By Lemma~\ref{lem:linesAndHyps} (i), if $X$ is a line then $r(X\cap Y)=0$. Hence $r(X'Y)=r(XY)=r(Y)+1$ for every non-trivial flat $X'\subset X$ and $r(Y|X')>r(Y|X)$, completing the proof for when $X$ is a line. The proof for when $Y$ is a hyperplane follows immediately from Lemma~\ref{lem:linesAndHyps} (ii).
\end{proof}

Finally, we have the $\dlp$ counterpart for Prop.~\ref{prop:AKSpecialRankGeneral}.

\begin{proposition}\label{prop:DLSpecialRankGeneral}
    Let $\Mat=(Q,r)$ be a sparse paving matroid of rank $k$ and let $X,Y\subseteq Q$ be a disjoint nonmodular pair of flats. If $r(X)+r(Y)=k+1$ and $X$ is not a circuit-hyperplane, then there is a quasi-intersection of $(X,Y)$ in $\Mat$.
\end{proposition}

\begin{proof}
    The proof analogously follows that of its $\akp$ counterpart, so some details are omitted. 
    
    If $Y$ is a hyperplane then Lemma~\ref{lem:DLGuaranteed} applies. Since $r(X\cap Y)=0$, then $|XY|=|X|+|Y|=k+1$ and $r(XY)=k$. If there is no circuit-hyperplane in $XY$ then Prop.~\ref{prop:XYCircuitDL} applies. We finish by considering the case where $XY$ contains a circuit-hyperplane.
    
    For all flats $X'\subset X$ such that $X'Y$ is not a circuit-hyperplane, $r(Y|X')=r(Y)>r(Y|X)$ by the independence of $X'$ and $Y$. 
    Now, let $X_1\subseteq X$ and $Y_1\subseteq Y$ be such that $X_1Y_1$ is a circuit-hyperplane. If $X_1\subset X$, then $r(X_1)=r(X)-1$ and $r(Y|X_1)=r(X_1Y)-r(X_1)=k-1-r(X)+1=r(Y|X)$. Therefore, $X_1$ is a quasi-intersection of $(X,Y)$ in $\Mat$.
\end{proof}

In general, we do not know if there's any link between $\dlp$ and $\akp$ properties, but in the special case of rank-4 matroids, we have the following.

\begin{proposition}\label{prop:DLIffAK}
    Let $\Mat=(Q,r)$ be a rank-4 matroid. Then it admits a $\dlp$ extension if and only if it admits an $\akp$ extension that is a matroid. 
\end{proposition}

\begin{proof}
    Firstly, by Lemmas~\ref{lem:DLGuaranteed},~\ref{lem:AKExtForLines}, and~\ref{lem:YHypAK}, we only need to concern ourselves with nonmodular pairs of flats $(X,Y)$ of $\Mat$ where $X$ is a hyperplane and $Y$ is a line. Let $\cL$ be the family of lines $X'\subset X$ satisfying that $X'Y$ is a hyperplane. 
    
    If $|\cL|$ is 0 (resp., 1), then $X$ (resp., $X'\in\cL$) is a quasi-intersection of $(X,Y)$, and $\cF_X$ (resp., $\cF_{X'}$) is the modular cut corresponding to an $\akp$ extension of $(X,Y)$. 

    Now suppose that $|\cL|>1$. Note that
    \begin{enumerate}
        \item[(i)] $r(Y|X')=r(X'|Y)=r(Y|X)=1$ for every $X'\in \cL$, and
        \item[(ii)] $r(Y|L)=r(L|Y)=2$ for every line $L\subset X$, $\not\in \cL$.
    \end{enumerate} 
    
    Suppose that $\Mat$ admits a $\dlp$ extension $\Mat'$. By (DL2) and (i), there exists $T$ such that $T\in\cl_{\Mat'}(X')$ for every $X'\in \cL$. Hence, $r(X'|T)=1$. Also, $r(L|T)=r(L)$ for every line $L\subset X$, $L\not\in \cL$, proving (AK2). Hence, $\Mat'$ is an $\akp$ extension for $(X,Y)$.

    Now suppose that $\Mat'$ is an $\akp$ extension, and let $z_o=\akp(X,Y)$. For every flat $X'\subset X$, it is clear that when $r(X'|z_o)<r(X')$ then $r(z_o|X')=0$ and $X'\in\cL$. On the other hand, if $r(X'|z_o)=r(X')$ then $r(z_o|X')=1$ and $X'$ is a line not in $\cL$ or $r(X')<2$. Hence, $z_o$ is a quasi-intersection of $(X,Y)$ and $\Mat'$ is also a $\dlp$ extension for $(X,Y)$.
\end{proof}

\section{Some New Non-Algebraic Matroids}
\label{sec: exps}
Using recursive applications of the $\dlp$ and $\akp$ properties, we were able to discover some new non-almost entropic and/or non-algebraic matroids on 9 and 10 points. 

Every almost entropic polymatroid $\Poly$ admits at least one $\akp$ extension for each pair of subsets $(X,Y)$ of its ground set \cite{BFP23,FKMP20}. Taking $k$ pairs of subsets, we can find recursively an entropic polymatroid that is a recursive $AK$ extension of $\Poly$.

Hence, to show that a matroid is not almost entropic, it is enough to find a sequence of pairs of sets for which there is no polymatroid extension of the original matroid that contains $\akp$-informations of these pairs of sets.


\begin{example}
    Let $A_1=\{0,1,2\}$, $A_2=\{3,4,5\}$, $A_3=\{6,7,8\}$, $B_1=\{0,3,6\}$, $B_2=\{1,4,7\}$, and $B_3=\{2,5,8\}$. The Tic-Tac-Toe matroid $(T_3)$ is the sparse paving $(5,9)$ matroid with ground set $Q=A_1A_2A_3$ and circuit-hyperplanes $\mathcal{C}(T_3)=\{A_iB_j:i,j\in\{1,2,3\}\text{ and }(i,j)\not=(2,2)\}$ \cite{AlHo95}. It is a non-linearly representable matroid that satisfies the Ingleton inequality\cite{AlHo95,BBFP21}. Its dual $(T_3^*)$ is not 3-$\akp$ (proved in Proposition 4.14 of~\cite{Bam21}) and is therefore not almost entropic nor algebraic. Let $\Mat=(Qe,r)$ be the identically self-dual (ISD), sparse paving $(5,10)$ matroid with circuit-hyperplanes $\mathcal{C}(T_3)\cup\{Ce:C\in\mathcal{C}(T_3^*)\}$. Note that $\Mat\setminus e=T_3$ and $\Mat/e=T_3^*$. Since $T_3$ is not $\cip$ and so not folded linear~\cite{BBFP21}, and $T_3$ is a minor of $\Mat$, then $\Mat$ is not folded linear. Applying $\akp$ at depth 3, we found that it is not almost entropic, and therefore, not algebraic using the following combinations:
   \begin{align*}
       \alpha &= \akp(4578e,36), &\beta &= \akp(1245e,03),\\
       \gamma &= \akp(258e\alpha\beta,17).
   \end{align*}
   The matroid $\Mat$ is a smallest ISD, Ingleton-compliant (i.e., satisfies the Ingleton inequality~\cite{Ing71}) matroid that is not almost entropic. Though it is 2-$\dlp$, we do not yet know if it is 3-$\dlp$.
    
    The question of whether the Tic-Tac-Toe matroid is algebraic remains open.
\end{example}

\begin{example}\label{exp: fam2}
    Consider the $(5,9)$ matroids with the following Bollen identifiers~\cite{Bollen18}: $100735$, $100736$, $100755$, $103147$, and $147269$. These matroids are Ingleton-compliant sparse paving matroids that are not 2-$\cip$. They are also Frobenius flock representable and satisfy Ingleton-Main at all depths. While they satisfy Dress-Lov\'asz up to depth 3, we do not yet know if they are 4-$\dlp$. In any case, they are neither almost entropic nor algebraic as we found that they fail $\akp$ at depth 4 using the following:
   \begin{align*}
       \alpha &= \akp(12678,03), &\beta &= \akp(03678\alpha,15),\\
       \gamma &= \akp(1257\alpha\beta,48), &\tau &= \akp(0357\alpha\gamma, 26). 
   \end{align*}
   These matroids are among the smallest Ingleton-compliant sparse paving matroids that are not almost entropic. Since they are from the family mentioned in~\cite[Section 4.5.2]{Bam21}, it is likely that many of those matroids will also not be almost entropic
\end{example}

\begin{example}
In~\cite{Bollen18}, Bollen found all matroids that are not algebraic due to failing the Ingleton-Main lemma at depths up to 5 for $(4,9)$ and $(5,9)$ matroids. Due to time constraints, he wasn't able to find all $(4,9)$ matroids that fail the same property at depth 6, leaving some unchecked. Going at these unchecked matroids, we were able to find all that are not 6-$\dlp$ and as well as those that are not 7-$\dlp$. And since $\dlp$ and the Ingleton-Main lemma are equivalent for rank-4 matroids, then these are also the remaining $(4,9)$ matroids that fail Ingleton-Main at depths 6 and 7.\fullv{ These matroids are listed in Appendix~\ref{sec: nonAlgbMats}.}\confv{ These matroids are listed in~\cite{BaFa24}.} Time constraints meant we were not able to do an exhaustive check for non-8-$\dlp$ matroids. However, by randomly selecting and testing, we were able to find a few non-8-$\dlp$ $(4,9)$ matroids. These are matroids $5635$, $7262$, and $103732$. A summary of these results is shown in Table~\ref{table:DLandIMmatroids}, which is an update on~\cite[Table 6]{Bollen18}. That is, we updated the values for $\dlp$ of depth 6,7, and 8, as discussed above. For depth 8, the search is not exhaustive, and so this number is a lower bound. 

\begin{longtblr}[
  caption = {Dress-Lov\'asz (DL) and Ingleton-Main (IM) check.},
  entry = {Short Caption},
  label = {table:DLandIMmatroids},
]{
  colspec = {l|ccc}, 
  row{1} = {font=\bfseries}, 
}
    \hline 
	$(r,n)$    & $(4,8)$ & $(4,9)$ & $(5,9)$ \\ \hline
	$\dlp$         & 39      & 27,137  & 27,137  \\
	$\dlp$ depth 2 & 39      & 27,137  & 27,137  \\
	IM depth 3 & 39      & 28,418  & 27,144  \\
	IM depth 4 & 39      & 30,171  & 27,442  \\
	IM depth 5 & 39      & 30,658  & 27,500  \\ 
	$\dlp$ depth 6 & 39      & 31,104  & ?  \\ 
	$\dlp$ depth 7 & 39      & 31,370  & ?  \\ 
	$\dlp$ depth 8 & 39      & $31,373^*$  & ?  \\ \hline
\end{longtblr}

In addition to being non-algebraic, we found matroid $129075$, a non-6-$\dlp$ matroid, to be non-almost entropic due to failing $\akp$ at depth 6 with the following combinations:
\begin{align*}
   \alpha &= \akp(2678, 35), &\beta &= \akp(2356\alpha, 14),\\
   \gamma &= \akp(01234\beta, 57), &\tau &= \akp(01234\beta\gamma, 67), \\ 
   \mu &= \akp(3468, \alpha\gamma), &\nu &= \akp(2678\alpha\tau, \beta\mu). 
\end{align*}
As for the other $(4,9)$ matroids we found to be non-algebraic, we do not yet know if they are almost entropic.

\end{example}

\section{Secret Sharing Schemes}
\label{sec: sss results}

We conclude this work presenting new results on secret sharing schemes obtained with the improvements of the $\akp$ technique presented in this paper. For an introduction to secret sharing, and more detailed definitions, see~\cite{BBFP21,Bei11,Pad12}.

In order to get information-theoretic lower bounds on the efficiency of these schemes, we consider the following definition. A \emph{secret sharing scheme} on a set $P=\{1,2,\ldots,n\}$ is a collection of discrete random variables $\Sigma=(S_0,S_1,\ldots,S_n)$ such that $H(S_0)>0$ and $H(S_0|S_P)=0$, where $H$ is the Shannon entropy and $S_0$ is the random variable associated to the \emph{dealer}, $p_0$.

We say that a subset $X\subseteq P$ is \emph{authorized} if $H(S_0|S_X)=0$, and we say that a subset is \emph{forbidden} if $H(S_0|S_X)=H(S_0)$. In this work, we only consider \emph{perfect} schemes, that is, schemes where every subset is either authorized or forbidden. The family of authorized subsets is called the \emph{access structure} of the scheme, and it is denoted by $\Gamma$. The \emph{information ratio} of the scheme is a measure of the scheme's efficiency given as $\max_i\{H(S_i)\}/H(S_0)$. If the information ratio is $1$, we say that the scheme is \emph{ideal}.

The access structure of ideal secret sharing schemes are ports of matroids~\cite{BrDa91}. However, the converse is not true: only ports of entropic matroids admit ideal schemes~\cite{BrDa91,Matus99,Sey92}. 

For a given access structure $\Acc$, the infimum of the information ratio of all schemes realizing $\Gamma$ is denoted as $\sigma(\Acc)$. Bounds on this value can be obtained using information inequalities. Csirmaz~\cite{Csi97} found a family of access structures $\{\Gamma_n\}_n$ satisfying that $\sigma(\Gamma_n)=\Omega(n/\log n)$. This is the best known lower bound for the information ratio. For matroid ports, finding non-trivial lower bounds requires using non-Shannon information inequalities, but until now all lower bounds that have been found are constant and smaller than 2~\cite{BBFP21,BLP08,FKMP20,Gurpinar22}.

Matroids that do not satisfy the $\akp$ property are not entropic. Therefore, they do not admit ideal schemes and the ports of such matroids will require schemes with information ratio greater than 1. 

In this work, we improved the linear programming problems introduced in~\cite{FKMP20} with the optimizations presented in Theorem~\ref{thm:AKRedefined} and in Definition~\ref{def:NewAK}. The resulting linear programming problem is presented below. We define $Q=Pp_0$. The conditions for  $(QZ, f)$ being compatible with $\Gamma$ are explained in \cite{BBFP21,FKMP20}.

\begin{lpp}
	\label{tab:LPAKNew}
	Let $X,Y\subseteq P$. The optimal value of this linear programming problem 
	is a lower bound on $\sigma(\Acc)$.
	\begin{align*}
	\text{Minimize }\quad& v \\
	\text{subject to}\quad& v \ge f(x)  \text{ for every } x \in P\\
    & (QZ, f) \text{ is a polymatroid compatible with }\Gamma\\
	& \mathrm{(AK1), (AK2')}\text{ on }Z\text{ and } (X,Y)
	\end{align*}
\end{lpp}

By solving LP~\ref{tab:LPAKNew} for the ports of the non-Ingleton-compliant matroids on 8 points, we were able to improve the bounds on $\sigma({\Acc})$ for a number of them. 
In the case of some of the ports of the $Q_8$ matroid, the new $\akp$ definition
(Definition~\ref{def:NewAK}) gave better bounds than the old one (see Section~\ref{sec: AK lemma}). Perhaps, this is an indication of the strength of the new definition over the previous one. Using this new definition, we now have that the current best bound on $\sigma({\Acc})$ for a matroid port is 52/45, which was obtained for some of the ports of the $AG(3,2)'$ matroid. Apart from the already mentioned ports of the $Q_8$ matroid which lower bound was gotten using 2-$\akp$, the other results here were gotten using 4-$\akp$. The bounds for the ports of the named matroids presented in Table~\ref{table:8ptstable} also match those gotten for the same matroid ports using the copy lemma~\cite{Gurpinar22}, evidencing that even without knowing if there are any symmetries inherent in the matroid, the $\akp$ lemma can still be used to get bounds that are on par with those gotten when symmetry conditions are taken into account in using the copy lemma. 

\begin{longtblr}[
  caption = {Improved lower bounds on $\sigma({\Acc})$ for some 8-point matroids. Best previous bounds were from~\cite{BBFP21} unless specified.},
  entry = {Short Caption},
  label = {table:8ptstable},
]{
  colspec = {ccXX}, hlines,
  row{1} = {font=\bfseries},
}
Matroid          & Port       & Previous Bound     & Improved Bound \\
	1490           & $0, 2, 3, 4, 5, 6$         & $8/7$ & $53/46$          \\
	1491           & $0, 3, 7$                  & $33/29$ & $8/7$          \\
	1491           & $2, 4, 6$               & $8/7$ & $84/73$             \\
	1492           & $0, 1, 2, 3, 6, 7$   & $49/43$ & $38/33$           \\
	1499           & $0, 2, 3, 4, 5, 6$         & $8/7$ & $38/33$         \\
	1500           & $0, 2, 4, 5$         & $8/7$ & $38/33$              \\
	1501           & $0, 1, 2, 3, 6, 7$        & $33/29$ & $8/7$  \\
	1502           & $2, 3, 4, 7$               & $33/29$ & $8/7$   \\
	1525           & $0, 2, 4, 5$               & $33/29$ & $8/7$   \\
	1526           & $0, 2, 3, 4, 5, 6$         & $8/7$ & $38/33$   \\
	1532           & $0, 1, 2, 3, 5, 6 $        & $33/29$ & $8/7$   \\
	1579           & $0, 2, 4, 5$               & $33/29$ & $8/7$  \\
	$AG(3,2)'$ & $1, 3, 5, 7$               & $49/43$ & $52/45$            \\
	$F_8$        & $3, 4, 5, 6$               & $23/20$ \cite{Gurpinar22} & $38/33$         \\
	$Q_8$        & $1, 4, 6, 7$               & $49/43$ & $8/7$           \\
\end{longtblr}

In addition to these 8-point matroids, we also present new bounds on $\sigma({\Acc})$ for ports of the 9-point matroids described in Example~\ref{exp: fam2}. For each of these matroids, there is at least a port for which $\sigma({\Acc})\geq89/88$. This was the best bound we got for ports of these matroids. However, since we only tried one combination of sets to get the bounds, we do not rule out the possibility that other combinations might produce better bounds. These are shown in Table~\ref{tab: 9ptstable}. 

\begin{table}[ht]
	\centering
	\caption{Bounds on $\sigma({\Acc})$ for some non-AK (5,9) matroids}
    \label{tab: 9ptstable}
\begin{tabular}{c|c|c}
    \hline
    \textbf{Matroid} & \textbf{Sets} & \textbf{Best Bound} \\ \hline
    100735 & \multirow{5}{*}{$\begin{array}{c}
            \{1, 2, 6, 7, 8\}, \{0, 3\} \\
            \{0, 3, 6, 7, 8, 9\}, \{1, 5\} \\
            \{1, 2, 5, 7, 9, 10\}, \{8, 4\} \\
            \{3, 5, 7, 0, 9, 11\}, \{2, 6\} \\
        \end{array}$} & \multirow{5}{*}{$1.011\bar{36}=89/88$} \\ \cline{1-1}
    100755 & & \\ \cline{1-1}
    100736 & & \\ \cline{1-1}
    103147 & & \\ \cline{1-1}
    147269 & & \\ \hline
\end{tabular}
\end{table}

\appendix

\fullv{
\section{Non-Algebraic Matroids}
\label{sec: nonAlgbMats}
The non-6-$\dlp$ matroids we found, excluding those found by Bollen, 
are listed in Table~\ref{table:nonDL6mats}, while those that are not 7-$\dlp$ are listed in Table~\ref{table:nonDL7mats}. 

\begin{longtblr}[
  caption = {Non-6-$\dlp$ $(4,9)$ Matroids.},
  entry = {Short Caption},
  label = {table:nonDL6mats},
]{
  colspec = {ccccccc}, hlines, vlines,
}
    6429 & 6546 & 6780 & 6823 & 6840 & 6909 & 7035 \\ 
    7066 & 7067 & 7110 & 7144 & 7146 & 7183 & 7186 \\ 
    7215 & 7240 & 7268 & 7285 & 7318 & 7327 & 7328 \\ 
    7366 & 7454 & 7498 & 7753 & 7756 & 7765 & 7807 \\ 
    7825 & 7828 & 7843 & 7953 & 7954 & 7955 & 8048 \\ 
    8049 & 9057 & 9072 & 9116 & 9224 & 9247 & 12949 \\
    18130 & 19749 & 19877 & 20483 & 20586 & 20800 & 21070 \\ 
    21071 & 30065 & 30998 & 34496 & 34498 & 39130 & 43892 \\ 
    43899 & 43920 & 45055 & 45143 & 46213 & 47625 & 47640 \\ 
    47643 & 49171 & 49174 & 55311 & 55344 & 55349 & 55352 \\ 
    55366 & 55418 & 55419 & 55494 & 55496 & 55564 & 55578 \\ 
    55595 & 55607 & 55613 & 55652 & 55854 & 57163 & 57585 \\ 
    58538 & 58548 & 59364 & 59379 & 59452 & 68407 & 70954 \\ 
    72045 & 72395 & 72563 & 72731 & 72733 & 72767 & 72843 \\ 
    72864 & 72879 & 73182 & 73238 & 73262 & 73893 & 74412 \\ 
    75017 & 75399 & 75405 & 75407 & 75408 & 78830 & 80380 \\ 
    81017 & 81270 & 81335 & 84491 & 88947 & 88980 & 88998 \\ 
    89007 & 89020 & 89029 & 89037 & 91469 & 91579 & 91846 \\ 
    92043 & 92676 & 93533 & 94039 & 95365 & 95404 & 95412 \\ 
    95436 & 95437 & 95438 & 95487 & 95495 & 95532 & 95654 \\ 
    95796 & 95908 & 95909 & 95941 & 96033 & 96330 & 96346 \\ 
    96364 & 96438 & 96446 & 96486 & 96592 & 97412 & 98136 \\ 
    98183 & 99690 & 100129 & 100141 & 100159 & 100977 & 101050 \\ 
    101418 & 102179 & 102537 & 104402 & 104881 & 104970 & 104983 \\ 
    105004 & 112932 & 113013 & 113458 & 113459 & 113817 & 114236 \\ 
    114719 & 115835 & 116132 & 116136 & 116137 & 117651 & 122780 \\ 
    129075 & 129076 & 129492 & 129495 & 139248 & 139307 & 139326 \\ 
    139328 & 139330 & 139369 & 139370 & 139389 & 139404 & 139409 \\ 
    139411 & 139418 & 139419 & 139445 & 139548 & 141113 & 141996 \\ 
    143899 & 144026 & 144285 & 144322 & 144353 & 144405 & 144410 \\ 
    144411 & 144414 & 144428 & 144471 & 144548 & 144914 & 145366 \\ 
    147282 & 147295 & 147388 & 147391 & 147756 & 147760 & 147804 \\ 
    147954 & 153825 & 153833 & 153849 & 154205 & 154602 & 154605 \\ 
    154606 & 154607 & 155113 & 155288 & 155929 & 155930 & 156105 \\ 
    156231 & 156280 & 156501 & 156569 & 156610 & 156638 & 157002 \\ 
    157049 & 157230 & 157949 & 158740 & 158814 & 158819 & 158874 \\ 
    159035 & 159071 & 159092 & 159155 & 159157 & 159170 & 159192 \\ 
    159195 & 159200 & 159203 & 159205 & 159215 & 159223 & 159225 \\ 
    159228 & 159230 & 159234 & 159239 & 159351 & 159412 & 159413 \\ 
    159834 & 159942 & 160551 & 160694 & 160727 & 160891 & 161038 \\ 
    161488 & 161556 & 161589 & 161593 & 161723 & 161729 & 161820 \\ 
    161838 & 161896 & 162062 & 162079 & 162507 & 163170 & 163695 \\ 
    169252 & 170268 & 171948 & 171958 & 172010 & 172035 & 173104 \\ 
    173910 & 173930 & 173943 & 174171 & 174818 & 174936 & 174942 \\ 
    174948 & 174965 & 174969 & 174976 & 174978 & 174985 & 175004 \\ 
    175015 & 175018 & 175027 & 175029 & 175034 & 175314 & 175432 \\ 
    175459 & 175979 & 176375 & 183588 & 183814 & 183828 & 183855 \\ 
    184062 & 184063 & 184338 & 186678 & 187231 & 187258 & 187269 \\ 
    187326 & 187443 & 187476 & 187478 & 187482 & 187613 & 187659 \\ 
    187690 & 187727 & 187816 & 188004 & 188351 \\ 
\end{longtblr}

\begin{longtblr}[
  caption = {All Non-7-$\dlp$ $(4,9)$ Matroids.},
  entry = {Short Caption},
  label = {table:nonDL7mats},
]{
  colspec = {ccccccc}, hlines, vlines,
}
    1244 & 6112 & 6122 & 6172 & 6173 & 6175 & 6185 \\ 
    6186 & 6188 & 6191 & 6196 & 6339 & 6345 & 6371 \\ 
    6379 & 6391 & 6412 & 6417 & 6462 & 6563 & 6784 \\ 
    6799 & 6828 & 6915 & 6935 & 7013 & 7039 & 7196 \\ 
    7245 & 7248 & 7253 & 7262 & 7314 & 7338 & 7354 \\ 
    7388 & 7417 & 7852 & 7922 & 7935 & 7959 & 8022 \\ 
    9069 & 9071 & 9108 & 9238 & 9239 & 9242 & 9249 \\ 
    9253 & 12250 & 15744 & 18610 & 21244 & 34469 & 34480 \\ 
    34483 & 34487 & 34488 & 34604 & 34605 & 34609 & 38800 \\ 
    38988 & 43869 & 47603 & 54910 & 55299 & 55316 & 55346 \\ 
    55426 & 55492 & 55566 & 55579 & 56698 & 56828 & 56894 \\ 
    59221 & 59456 & 66529 & 66560 & 70937 & 72573 & 73430 \\ 
    74350 & 75453 & 81302 & 81319 & 82110 & 87912 & 88770 \\ 
    91754 & 92566 & 94389 & 95363 & 95418 & 95431 & 95432 \\ 
    95579 & 95580 & 95600 & 95659 & 95872 & 95982 & 95990 \\ 
    95999 & 96022 & 96054 & 96059 & 96060 & 96064 & 96074 \\ 
    96091 & 96094 & 96132 & 96445 & 98080 & 98186 & 99199 \\ 
    100135 & 100271 & 100272 & 100747 & 100773 & 100789 & 100802 \\ 
    100811 & 100833 & 100917 & 100919 & 100921 & 100935 & 100936 \\ 
    100949 & 100961 & 100963 & 100971 & 101068 & 101127 & 101236 \\ 
    101258 & 101290 & 101352 & 101357 & 101369 & 101441 & 101655 \\ 
    101659 & 101735 & 101793 & 101875 & 101924 & 101933 & 101937 \\ 
    101961 & 102013 & 102041 & 102132 & 102656 & 102660 & 103213 \\ 
    103216 & 103226 & 104435 & 104965 & 104997 & 104999 & 105003 \\ 
    116285 & 116301 & 117873 & 125163 & 125591 & 126686 & 128191 \\ 
    129790 & 134729 & 135411 & 135417 & 135421 & 135424 & 136472 \\ 
    139406 & 139412 & 139578 & 139603 & 141988 & 144010 & 144156 \\ 
    144301 & 144426 & 147291 & 147294 & 147296 & 147322 & 154063 \\ 
    154194 & 154200 & 155301 & 156009 & 156060 & 156074 & 156116 \\ 
    156117 & 156536 & 156558 & 158736 & 158763 & 158826 & 158827 \\ 
    158875 & 158889 & 158890 & 158911 & 158913 & 158923 & 158947 \\ 
    159025 & 159109 & 159411 & 159421 & 159499 & 159515 & 159550 \\ 
    159607 & 159828 & 159844 & 159939 & 159987 & 160323 & 160353 \\ 
    160354 & 160368 & 160444 & 160558 & 160800 & 160990 & 161089 \\ 
    161467 & 161664 & 161735 & 171823 & 171986 & 173245 & 175041 \\ 
    175615 & 176416 & 178706 & 182843 & 182920 & 182945 & 183016 \\ 
    183092 & 183146 & 183327 & 183374 & 183407 & 183640 & 183646 \\ 
    183823 & 183881 & 183889 & 184046 & 187591 & 187952 & 188422 \\ 
\end{longtblr}

\section{Pseudomodularity}
\label{sec: pseudomod}

In this section, we look at one more tool that can be used to detect non-algebraic matroids. Similar to $\dlp$, it is also an extension property of algebraic matroids. While it plays the same role as $\dlp$ when dealing with rank-4 matroids, the 2 properties are not necessarily the same. 

\begin{definition}\cite[Theorem 1.4]{BjLo86}\label{def:pseu}
A matroid $\Mat=(Q,r)$ with a lattice $\mathcal{L}$ is said to be \emph{pseudomodular} if, for every $X,Y,Z\in\mathcal{L}$ such that $X$ covers $X\cap Z$ and $Y$ covers $Y\cap Z$, then 
\[r(X\cap Y)-r(X\cap Y\cap Z)\leq 1.\]
\end{definition}

\begin{theorem}[\cite{DrLo87}]\label{th:amp}
Full algebraic matroids are pseudomodular.
\end{theorem}
As a consequence of this, we have the following necessary condition for a matroid to be algebraic.

\begin{proposition}\label{pr:pseu}
    Let $\Mat=(Q,r)$ be an algebraic matroid, let $X,Y,Z$ be flats of $\Mat$, and let $A=X\cap Y\cap Z$. If these satisfy the following conditions
    \begin{enumerate}
    \item $X$ covers $X\cap Z$, $Y$ covers $Y\cap Z$, and
    \item $r(X\cap Y)-r(A)> 1$,    
    \end{enumerate} 
    then there exists an algebraic matroid $\Mat'=(Qe,r)$ that extends $\Mat$ by a point $e$ that is in the closure of $X$, $Y$ and $Z$ and $r(Ae)>r(A)$.
\end{proposition}

\begin{proof}
Let $\mathcal{N}=(Q',r)$ be a full algebraic matroid that is an extension of $\Mat$. 
Let $X',Y',Z'$ be the closures in $\mathcal{N}$ of $X,Y,Z$. By Theorem~\ref{th:amp}, $\mathcal{N}$ is pseudomodular and satisfies
\[r(X'\cap Y'\cap Z')\geq r(X'\cap Y')-1\geq r(X\cap Y)-1>r(A)\]
Therefore, there is a $p\in Q'$ that, in $\mathcal{N}$ lies in the closures of $X$, $Y$ and $Z$ that is not in the closure of $A$. Then, we can extend $\Mat$ with $p$, and the resulting matroid is still algebraic.
\end{proof}

From Proposition~\ref{pr:pseu}, we get a tool to use in finding non-algebraic matroids. Given a matroid $\Mat$ and some depth $d>0$, it first checks if $\Mat$ has \emph{pseudotriples} $(X,Y,Z)$, that is, triples of flats satisfying the conditions in Proposition~\ref{pr:pseu}. If so, it looks for a matroid extension guaranteed by the proposition. If no such extension exists, then the matroid is not $d$-pseudomodular, and therefore not algebraic. If there exist such matroid extensions, then it checks if there is at least one of these matroids that is $d-1$-pseudomodular in the same way. This is done recursively for every pseudotriple at every depth until depth 1. 

\subsection{Algorithms for Pseudomodularity}
\label{sec:algorithms}
The algorithms we used in finding non-pseudomodular matroids are shown in this section.

We show how we generate pseudotriples for a matroid. This step is in two parts. First, we generate all triples of flats using Algorithm~\ref{alg:getTriples}, then we validate all generated triples using Algorithm~\ref{alg:validateTriples}, returning only the ones that can be pseudotriples. Note that instead of waiting to get all pseudotriples, it is more efficient to yield each one as it is generated. Since only the first two sets in any pseudotriple are symmetric, we ensure that we do not miss out on any possible pseudotriple by swapping the place of $X$ and $Z$, as well as $Y$ and $Z$ once each.

Next, we use Algorithm~\ref{alg:getExtensions} to find extensions of the given matroid by an element that lies in the intersection of the sets in the pseudotriple. At depth 1, which is the base depth, we check if the matroid is pseudomodular using Algorithm~\ref{alg:baseCheck}. And finally, as the entry point to the program, we have Algorithm~\ref{alg:recursiveCheckPM}.

\begin{algorithm}
\begin{algorithmic}
\caption{Get Pseudotriples}\label{alg:getTriples}
\Function{getPSMTriples}{matroid}
    
    \State $validTriples \gets Array$
    \For{$X, Y, Z$ \textbf{in} $Flats$}
        \If{\Call{validateTriples}{$X,Y,Z$}}   
            \State \textbf{add} $(X,Y,Z)$ \textbf{to} $validTriples$
        \ElsIf{\Call{validateTriples}{$X,Z,Y$}}
            \State \textbf{add} $(X,Z,Y)$ \textbf{to} $validTriples$
        \ElsIf{\Call{validateTriples}{$Z,Y,X$}}
            \State \textbf{add} $(Z,Y,X)$ \textbf{to} $validTriples$
        \EndIf
    \EndFor
    \State \textbf{return} $validTriples$
\EndFunction
\end{algorithmic}
\end{algorithm}

\begin{algorithm}
\begin{algorithmic}
\caption{Validate Pseudotriples}
\label{alg:validateTriples}
\Function{validateTriples}{X,Y,Z}
    
    \If{$X$ covers $X \cap Z$ \textbf{and} $Y$ covers $Y \cap Z$ \textbf{and} $r(X \cap Y) - r(X \cap Y \cap Z) > 1$}
        \State \textbf{return True} 
    \EndIf
    \State \textbf{return False} 
\EndFunction
\end{algorithmic}
\end{algorithm}

\begin{algorithm}
\begin{algorithmic}
\caption{Get Matroid Extensions}
\label{alg:getExtensions}
\Function{getExtensions}{$M$, subsets}

    \State $element \gets \text{String}(M.size())$
    \For{$N$ \textbf{in} $matroidExtensions(element, subsets)$}
        \State \textbf{return} $N$,$\,element$
    \EndFor
\EndFunction
\end{algorithmic}
\end{algorithm}

\begin{algorithm}
\begin{algorithmic}
\caption{Base Pseudomodular Check}
\label{alg:baseCheck}
\Function{baseCheckPSM}{$M$}
    
    \For{$triple$ \textbf{in} \Call{getPSMTriples}{$M$}}
        \State $X,Y,Z \gets$ $triple$
        \State $T \gets$ $X\cap Y\cap Z$
        \If{$T\in$ $modularCut(triple)$}
            \State \textbf{return} \textbf{false}
        \EndIf
    \EndFor
    \State \textbf{return} \textbf{true}
\EndFunction
\end{algorithmic}
\end{algorithm}

\begin{algorithm}
\begin{algorithmic}[0]
\caption{Main Recursive Function}
\label{alg:recursiveCheckPM}
\Function{recursivePSM}{$M$, depth}
    
    \If{$depth == 1$}
        \State $result \gets$ \Call{baseCheckPSM}{$M$}
        \State \textbf{return} $result$
    \EndIf
    \For{$triple$ \textbf{in} \Call{getPSMTriples}{$M$}}
        \State $X,Y,Z \gets$ $triple$
        \State $T \gets$ $X\cap Y\cap Z$

        \State $isPM \gets$ \textbf{false}
        \For{$N,e$ \textbf{in} \Call{getExtensions}{$M$, $triple$}}
            \If{$N.rank(T) \not=N.rank(Te)$ }
                \If{\Call{recursivePSM}{$N, depth-1$}}
                    \State $isPM \gets$ \textbf{true}
                    \State \textbf{break}
                \EndIf
            \EndIf
        \EndFor
        \If{$isPM$ \textbf{is false}}
            \State \textbf{return false}
        \EndIf
    \EndFor
    \State \textbf{return true}
\EndFunction
\end{algorithmic}
\end{algorithm}
}

\end{document}